\documentclass[12pt,a4paper]{amsart}
 \usepackage[titletoc,toc,title]{appendix}
\usepackage{t1enc}
\usepackage[latin1]{inputenc}
\usepackage{amssymb}
\usepackage{amsthm}
\usepackage{amsmath}
\usepackage{latexsym}
\usepackage{array}
\usepackage{mathrsfs}
\usepackage[all]{xy}
\usepackage{graphics}
\usepackage{xcolor}
\usepackage{bbm}
\usepackage{stmaryrd}
\usepackage{color}
\usepackage{enumerate}
\usepackage{tikz}
\usetikzlibrary{matrix}

\title{Symmetric correspondences on quadrics}
\date{14 April 2017}
\author{Rapha\"{e}l Fino}
\address
{Instituto de Matem\'{a}ticas \\
Ciudad Universitaria\\
UNAM\\
DF 04510, M\'{e}xico}
\address
{{\it Web page:}
{\tt http://www.matem.unam.mx/fino}}
\email {fino {\it at} im.unam.mx}
\numberwithin{equation}{section}
\theoremstyle{definition}

\newtheorem{rem}[equation]{Remark}
\newtheorem{lemme}[equation]{Lemma}

\newtheorem{prop}[equation]{Proposition}
\newtheorem{thm}[equation]{Theorem}
\newtheorem{cor}[equation]{Corollary}

\begin{document}

\begin{abstract}
We prove a result comparing the rationality of some \textit{elementary}
algebraic cycles introduced by Alexander Vishik, defined on orthogonal grassmannians, with the rationality
of some algebraic cycles defined on fiber products of the corresponding quadric.

\end{abstract}

\maketitle

\tableofcontents

\section{Introduction}


Let $X$ be a smooth projective quadric of dimension $n$ over a field $F$ associated with a non-degenerate $F$-quadratic form $q$ and let $K/F$ be a splitting field extension of $q$. Let us denote $[n/2]$ as $d$. 
In his paper \cite{uINV} dedicated to the Kaplansky's conjecture on the $u$-invariant of a field, Alexander Vishik 
introduced the so-called \textit{Elementary Discrete Invariant}.
This invariant encodes information about rationality of algebraic cycles on the $d+1$ orthogonal
grassmannians associated with $q$, by means of certain \textit{elementary classes} defined on these grassmannians over $K$: it is the collection
of the codimensions of those elementary classes that are already defined at the level of the base field
$F$ (such an algebraic cycle is called \textit{rational}).

In the current note, we relate the rationality of the \textit{highest} elementary classes (highest in the sense that, for each grassmannian, the highest elementary class is the one with maximal codimension)
with the rationality of certain algebraic cycles defined on some fiber products $X\times X\times \cdots \times X$.

More precisely, for any $I\subset \{0,\dots, d\}$, we write $\mathcal{F}(I)$ for the partial orthogonal
flag variety associated with $q$. So, for any $i\in \{0,\dots, d\}$, the variety $\mathcal{F}(i)$ is the grassmannian 
of $i$-dimensional totally isotropic subspaces and we denote it by $G_i$. In particular $G_0$ is the quadric $X$. 
For $J\subset I$, we write $\pi$ with subindex $I$ with $J$ underlined inside it
for the natural projection $\mathcal{F}(I)\rightarrow \mathcal{F}(J)$. In particular, for any $i\in \{0,\dots, d\}$, one can consider
\[\xymatrix{
X & \mathcal{F}(0,i) \ar[l]^{\pi_{(\underline{0},i)}} \ar[r]_{\pi_{(0,\underline{i})}} & G_i,
}\]
and we set 
\[Z^{i}_{n-i}:={\pi_{(0,\underline{i})}}_{\ast} \circ \pi_{(\underline{0},i)}^{\ast}(l_0)\in \text{CH}^{n-i}({G_i}_{K}),\]
where $l_0\in \text{CH}_0(X_{K})$
is the class of a closed point of $X_{K}$ of degree $1$ and $\text{CH}$ stands for the Chow ring
with $\mathbb{Z}$-coefficients.
The cycle $z^{i}_{n-i}:=Z^{i}_{n-i}\;(\text{mod}\;2) \in \text{Ch}^{n-i}({G_i}_{K})$ is the highest elementary class for the grassmannian $G_i$ (with $\text{Ch}$ the Chow ring
with $\mathbb{Z}/2\mathbb{Z}$-coefficients).
A.\,Vishik proved in \cite[Proposition 2.5]{uINV} that the rationality of $z^{i}_{n-i}$ implies the rationality of $z^{j}_{n-j}$ for any $j>i$.

We now introduce the other type of algebraic cycles coming into play in the main statement of the present paper (Theorem 1.1). For any $i\in \{0,\dots, d\}$, let us denote by 
$\text{sym}: \text{CH}^{\ast}(X^{i+1})\rightarrow \text{CH}^{\ast}(X^{i+1})$
the homomorphism $\Sigma_{s \in S_{i+1}}s_{\ast}$, where
$s:X^{i+1}\rightarrow X^{i+1}$ is the isomorphism associated with
a permutation $s$. We set 
\[\rho_i:=\text{sym}\left((\times_{j=0}^{i-1}h^j)\times l_0\right)\in \text{CH}^{n+i(i-1)/2}(X_{K}^{i+1}),\]
where $\times$ is the external product
and $h^j$ is the $j$-th power of the hyperplane section
class $h\in \text{CH}^1(X)$ (always rational).
Note that $\rho_1$ is
the \textit{ Rost correspondence} $1\times l_0+l_0\times 1$ of $X$ in the sense of \cite[\S 80]{EKM}
(we refer to \cite[\S 62]{EKM} for an introduction to correspondences). Note that the rationality of $\rho_i$ implies the rationality
of $\rho_j$ for any $j>i$. Note also that $\rho_0=l_0=Z^0_n$.

\medskip

The main result of this paper is the following statement

\begin{thm}[\textbf{Main Theorem}]
 \textit{Let $i \in \{0,\dots, d\}$. The cycle $z^{i}_{n-i}$ is rational if and only if the cycle} $\rho_i\;(\text{mod}\;2)$ \textit{is rational.}
\end{thm}

The existence of an algebraic cycle on $X^{i+1}$ satifying such an equivalence is predicted by the motivic decomposition \cite[Corollary 91.8]{EKM}.
Theorem 1.1 reduces certain questions about rationality of algebraic cycles on orthogonal grassmannians to the sole level of quadrics.
For example, it allows one to reformulate Vishik Conjecture \cite[Conjecture 3.11]{GPQ}
and this reformulation may help prove the conjecture.

The proofs in the present note mainly rely on computations of compositions of correspondences and the use of Chern classes of vector bundles over orthogonal grassmannians.

In section 2, we introduce some materials which will be required to prove {Theorem 1.1} in sections 3 and 5.
In section 4, we apply Theorem 1.1 (actually, only the direct implication is needed) to relate the first Witt index
of the quadric with the rationality of the cycle $z^{i}_{n-i}$ (Proposition 4.1).

\smallskip
\noindent
{\sc Acknowledgements.} I thank Nikita Karpenko and Alexander Vishik for stimulating discussions.

\section{Materials}

In this section, we use notations introduced in the introduction.

\subsection{Rational cycles on powers of quadrics}

We refer to \cite[\S 68]{EKM} for an introduction to cycles on powers of quadrics.
For any $i \in \{1,\dots, d\}$, we set
\[\Delta_i:=\text{sym}\left((\times_{j=1}^{i-1}h^j)\times 1\times l_0\right)
+\sum_{k=i}^d\text{sym}\left((\times_{j=1}^{i-1}h^j)\times h^k \times l_k\right)
\in \text{CH}^{n+i(i-1)/2}(X_{K}^{i+1}).\]
where $l_k$ is the class in $\text{CH}_k(X_K)$ of a $k$-dimensional totally isotropic subspace
of $\mathbb{P}\left((V_q)_K\right)$ (with $V_q$ the $F$-vector space associated with $q$).
If $n=2d$, we choose an orientation $l_d$ of the quadric.

The following observation is crucial for our matter (it will be used in Corollary 3.15 to obtain the first part Theorem 1.1).

\begin{lemme}
\textit{For any $i \in \{1,\dots, d\}$, the cycle} $\Delta_i\;(\text{mod}\;2)\in \text{Ch}^{n+i(i-1)/2}(X_{K}^{i+1})$
\textit{is rational.}
\end{lemme}

\begin{proof} We use an induction on $i$. 
In $\text{Ch}^{n}(X_{K}^{2})$, the cycle $\Delta_1\;(\text{mod}\;2)$
or $\Delta_1\;(\text{mod}\;2)+h^d\times h^d$, depending on whether $l_d^2=0$ or not, is the class of the diagonal. 
Hence, $\Delta_1\;(\text{mod}\;2)$ is rational.
Assume $\Delta_{i-1}\;(\text{mod}\;2)$ rational and let $\sigma \in S_{i+1}$ be a cyclic permutation ($i\geq 2$).
Then the cycle
\begin{equation}\sum_{j=0}^{i}\sigma_{\ast}^{j}(\Delta_{i-1}\times h^{i-1})=
\text{sym}\left((\times_{j=1}^{i-1}h^j)\times 1\times l_0\right)
+\sum_{k=i-1}^d\text{sym}\left((\times_{j=1}^{i-1}h^j)\times h^k \times l_k\right)\end{equation}
is congruent modulo $2$ to a rational cycle.
Moreover, the integral cycle 
\begin{equation} \text{sym}\left((\times_{j=1}^{i-1}h^j)\times h^{i-1} \times l_{i-1}\right)\end{equation}
is always rational (and divisible by $2$). Indeed, since $h^{i-1}$ appears exactly two times in every summand of (2.3), each distinct summand appears exactly two times, therefore it can be rewritten as
\[2\sum_{s\in A_{i+1}}s_{\ast}((\times_{j=1}^{i-1}h^j)\times h^{i-1} \times l_{i-1}),\]
and one has $2l_{i-1}= h^{n-i+1}$. Since $\Delta_i$ is the difference of (2.2) and (2.3), one get the 
conclusion.
\end{proof}

\subsection{Cycles on orthogonal flag varieties}

For $0\leq i\leq d$ and $n-i-d\leq j \leq n-i$, we set
\[Z^i_j:={\pi_{(0,\underline{i})}}_{\ast} \circ \pi_{(\underline{0},i)}^{\ast}(l_{n-i-j})\in \text{CH}^{j}({G_i}_{K})\]
and $z^i_j:=Z^i_j\;(\text{mod}\;2) \in \text{Ch}^{j}({G_i}_{K})$ 
(denoted respectively as $Z^{\text{\boxed{i-d}}}_j$ and $z^{\text{\boxed{i-d}}}_j$ in \cite{uINV}). 
The cycles $z^i_j$ are all the elementary classes defining
the Elementary Discrete Invariant $EDI(X)$. 


For any nonnegative integer $j$ and $i>0$  such that $j+i\leq d$, we set
\[W_j^i:={\pi_{(0,\underline{i})}}_{\ast} \circ\pi_{(\underline{0},i)}^{\ast}(h^{j+i})\in \text{CH}^j(G_i),\]
$W^0_j:=h^j$ and $w^i_j:=W^i_j\;(\text{mod}\;2)\in \text{Ch}^{j}(G_i)$
(denoted respectively as $W_j^{\text{\boxed{i-d}}}$ and $w_j^{\text{\boxed{i-d}}}$ in \cite{uINV}). 

\medskip

The proofs in the next sections will use the following lemma, which can easily be deduced from \cite[Propositon 2.1 and Lemma 2.6]{uINV} and its proofs.
For $0\leq i\leq d$, let us denote by $T_i$ the tautological vector bundle on $G_i$,
i.e. $T_i$ is given by the closed subvariety of the trivial bundle
$V\mathbbm{1}=V_q\times G_i$ consisting of pairs $(u,U)$ such that $u\in U$.
For a vector bundle $E$ over a scheme, we write $c_i(E)$ for the $i$-th Chern class with value in $\text{CH}$.


\begin{lemme}[Vishik]
\textit{One has} 
\begin{enumerate}[(i)]
\item ${\pi_{(0,\underline{i})}}_{\ast}\circ \pi_{(\underline{0},i)}^{\ast}(h^i)=W^i_0=[G_i];$
\item  \[\pi_{(\underline{i-1},i)}^{\ast}(Z_{j}^{i-1})=c_1(\mathcal{O}(1))\cdot \pi_{(i-1,\underline{i})}^{\ast}(Z_{j-1}^{i})+\pi_{(i-1,\underline{i})}^{\ast}(Z_{j}^{i}),\] \textit{where $\mathcal{O}(1)$ is the standard sheaf on the projective bundle
$\mathcal{F}(i-1,i)=\mathbb{P}_{G_{i-1}}(T^{\vee}_{i})$, with $T^{\vee}_{i}$ the vector bundle dual to  $T_{i}$;}
\item \textit{For $0\leq j < d-i+1$,} \[\pi_{(\underline{i-1},i)}^{\ast}(W_{j}^{i-1})=c_1(\mathcal{O}(1))\cdot \pi_{(i-1,\underline{i})}^{\ast}(W_{j-1}^{i})+\pi_{(i-1,\underline{i})}^{\ast}(W_{j}^{i});\]
\item  \[\pi_{(\underline{i-1},i)}^{\ast}(W_{d-i+1}^{i-1})=c_1(\mathcal{O}(1))\cdot \pi_{(i-1,\underline{i})}^{\ast}(W_{d-i}^{i})+2\pi_{(i-1,\underline{i})}^{\ast}(Z_{d-i+1}^{i}).\]
\end{enumerate}
\end{lemme}

The following statement is a direct consequence of the previous lemma. 

\begin{lemme}
 \textit{For any $1\leq i\leq d$, $i\leq k \leq d$ and $0\leq m\leq k$, one has}
\[ {\pi_{(\underline{i-1},i)}}_{\ast} \circ \pi_{(i-1,\underline{i})}^{\ast}(W^i_{k-i}\cdot Z^i_{n-i-m})
= \sum_{j=\text{max}(i-m,0)}^{\text{min}(k-m,i)} W^{i-1}_{k-m-j}\cdot {\pi_{(\underline{i-1},i)}}_{\ast} \circ \pi_{(i-1,\underline{i})}^{\ast}
(Z^{i}_{n-2i+j}).\]
\end{lemme}

\begin{proof}
We proceed by induction on $k$. The base of the induction $k=i$ is obvious. Let $i<k\leq d$
and $0\leq m\leq k$.
By Lemma 2.4(iii), one has 
\[\pi_{(i-1,\underline{i})}^{\ast}(W^i_{k-i}\cdot Z^i_{n-i-m})
=\left(\pi_{(\underline{i-1},i)}^{\ast}(W_{k-i}^{i-1})-c_1(\mathcal{O}(1))\cdot \pi_{(i-1,\underline{i})}^{\ast}(W_{k-i-1}^{i})\right) \cdot\pi_{(i-1,\underline{i})}^{\ast}(Z^i_{n-i-m}).\]
Hence, by the Projection Formula, the cycle ${\pi_{(\underline{i-1},i)}}_{\ast} \circ \pi_{(i-1,\underline{i})}^{\ast}(W^i_{k-i}\cdot Z^i_{n-i-m})$ is equal to
\[W^{i-1}_{k-i}\cdot {\pi_{(\underline{i-1},i)}}_{\ast} \circ \pi_{(i-1,\underline{i})}^{\ast}
(Z^{i}_{n-i-m})-
{\pi_{(\underline{i-1},i)}}_{\ast}\left(\pi_{(i-1,\underline{i})}^{\ast}(W_{k-i-1}^{i})
\cdot c_1(\mathcal{O}(1))\cdot \pi_{(i-1,\underline{i})}^{\ast}
(Z^{i}_{n-i-m}) 
\right).\]
For $m=0$, one has $c_1(\mathcal{O}(1))\cdot \pi_{(i-1,\underline{i})}^{\ast}
(Z^{i}_{n-i-m}) =\pi_{(\underline{i-1},i)}^{\ast}(Z^{i-1}_{n-i-m+1})$ by Lemma 2.4(ii).
Therefore, it follows from the Projection Formula that 
${\pi_{(\underline{i-1},i)}}_{\ast} \circ \pi_{(i-1,\underline{i})}^{\ast}(W^i_{k-i}\cdot Z^i_{n-i})$ can be rewritten as 
\[W^{i-1}_{k-i}\cdot {\pi_{(\underline{i-1},i)}}_{\ast} \circ \pi_{(i-1,\underline{i})}^{\ast}
(Z^{i}_{n-i})-Z^{i-1}_{n-i+1}\cdot 
{\pi_{(\underline{i-1},i)}}_{\ast}\circ \pi_{(i-1,\underline{i})}^{\ast}(W_{k-i-1}^{i})
.\]
Moreover, the cycle ${\pi_{(\underline{i-1},i)}}_{\ast}\circ\pi_{(i-1,\underline{i})}^{\ast}(W_{k-i-1}^{i})$ is trivial by dimensional reasons.
Thus, for $m=0$, one obtains the desired formula.
Assume that $m>0$. Then, by Lemma 2.4(ii),
one has $c_1(\mathcal{O}(1))\cdot \pi_{(i-1,\underline{i})}^{\ast}
(Z^{i}_{n-i-m}) =\pi_{(\underline{i-1},i)}^{\ast}(Z^{i-1}_{n-i-m+1})-\pi_{(i-1,\underline{i})}^{\ast}(Z^{i}_{n-i-m+1})$. 
Therefore, it follows from the Projection Formula and the same dimensional considerations as in the case 
$m=0$, that the cycle ${\pi_{(\underline{i-1},i)}}_{\ast} \circ \pi_{(i-1,\underline{i})}^{\ast}(W^i_{k-i}\cdot Z^i_{n-i-m})$ can be rewritten as
\[W^{i-1}_{k-i}\cdot {\pi_{(\underline{i-1},i)}}_{\ast} \circ \pi_{(i-1,\underline{i})}^{\ast}
(Z^{i}_{n-i-m})+
{\pi_{(\underline{i-1},i)}}_{\ast}\circ \pi_{(i-1,\underline{i})}^{\ast}(W_{(k-1)-i}^{i}\cdot Z^{i}_{n-i-(m-1)})
.\]
By the induction hypothesis, the second summand of the latter sum is equal to
\[\sum_{j=\text{max}(i-m+1,0)}^{\text{min}(k-m,i)} W^{i-1}_{k-m-j}\cdot {\pi_{(\underline{i-1},i)}}_{\ast} \circ \pi_{(i-1,\underline{i})}^{\ast}
(Z^{i}_{n-2i+j}).\]
Since ${\pi_{(\underline{i-1},i)}}_{\ast} \circ \pi_{(i-1,\underline{i})}^{\ast}
(Z^{i}_{n-i-m})=0$ if $m>i$ (by dimensional reasons),
one obtains the desired formula for ${\pi_{(\underline{i-1},i)}}_{\ast} \circ \pi_{(i-1,\underline{i})}^{\ast}(W^i_{k-i}\cdot Z^i_{n-i-m})$.
\end{proof}

Finally, the following description of the Chern classes modulo $2$ of the tautological vector bundles on orthogonal grassmannians will be needed.

\begin{lemme}
\textit{For any $1\leq i \leq d$ and $0\leq j \leq i$, one has }
\[c_j(T_{i-1})\;(\text{mod}\;2)={\pi_{(\underline{i-1},i)}}_{\ast} \circ \pi_{(i-1,\underline{i})}^{\ast}
(z^{i}_{n-2i+j})
\]
\textit{in} $\text{Ch}^j({G_{i-1}}_K)$.
\end{lemme}

\begin{proof}
We use an induction on $j$.
Since $Z^i_{n-2i}={\pi_{(0,\underline{i})}}_{\ast}\circ \pi_{(\underline{0},i)}^{\ast}(l_i)$, the cycle 
$\pi_{(i-1,\underline{i})}^{\ast}(Z^i_{n-2i})$
is the class in $\text{CH}\left( \mathcal{F}(i-1,i)_K  \right)$ of the closure of 
$\{(y,<y,x>)| x\in (y^{\perp}\backslash y) \cap L_i  \}$ in ${\mathcal{F}(i-1,i)_K}$,
with $L_i$ a fixed $i$-dimensional totally isotropic subspace of $\mathbb{P}\left( (V_q)_K\right)$.
Consequently, one has 
${\pi_{(\underline{i-1},i)}}_{\ast} \circ \pi_{(i-1,\underline{i})}^{\ast}(Z^{i}_{n-2i})=1$. 
Let $1\leq j \leq i$.
By \cite[Proposition 2.1]{uINV}, for any $l>d-i+1$, the element $c_l(V\mathbbm{1}/T_{i-1})$ is divisible by
$2$ in $\text{CH}^l({G_{i-1}}_K)$. Therefore, by Whitney Sum Formula (\cite[Proposition 54.7]{EKM}), 
the Chern class $c_j(T_{i-1})$ is congruent modulo $2$ to
\[c_j(V\mathbbm{1}/T_{i-1})+ \sum_{k=\text{max}(1,j-d+i-1)}^{j-1}  c_{j-k}(V\mathbbm{1}/T_{i-1})\cdot c_k(T_{i-1}).\] 
By the induction hypothesis and \cite[Proposition 2.1]{uINV}, the latter cycle  is congruent modulo $2$ to
\begin{equation} 
c_j(V\mathbbm{1}/T_{i-1})+ \sum_{k=\text{max}(1,j-d+i-1)}^{j-1} W^{i-1}_{j-k} \cdot
{\pi_{(\underline{i-1},i)}}_{\ast} \circ \pi_{(i-1,\underline{i})}^{\ast}
(Z^{i}_{n-2i+k}).
\end{equation}
Moreover, by Lemma 2.5, one has 
\begin{equation}
 {\pi_{(\underline{i-1},i)}}_{\ast} \circ \pi_{(i-1,\underline{i})}^{\ast}(W^i_{d-i}\cdot Z^i_{n-i-d+j})
= \sum_{k=\text{max}(j-d+i,0)}^{j} W^{i-1}_{j-k}\cdot {\pi_{(\underline{i-1},i)}}_{\ast} \circ \pi_{(i-1,\underline{i})}^{\ast}
(Z^{i}_{n-2i+k}).
\end{equation}

If $j-d+i-1<0$, combining (2.7), (2.8), \cite[Proposition 2.1]{uINV} and the base of the induction,
one get that $c_j(T_{i-1})$ is congruent modulo $2$ to
\[{\pi_{(\underline{i-1},i)}}_{\ast} \circ \pi_{(i-1,\underline{i})}^{\ast}
(Z^{i}_{n-2i+j})+
 {\pi_{(\underline{i-1},i)}}_{\ast} \circ \pi_{(i-1,\underline{i})}^{\ast}(W^i_{d-i}\cdot Z^i_{n-i-d+j}).\]
Furthermore, by applying Lemma 2.4(iv) after Lemma 2.4(ii) and using the Projection Formula combined with dimensional considerations (recall that $j>0$, so $i<d$), one
obtains 
\begin{equation} {\pi_{(\underline{i-1},i)}}_{\ast} \circ \pi_{(i-1,\underline{i})}^{\ast}(w^i_{d-i}\cdot z^i_{n-i-d+j})=
w^{i-1}_{d-i+1}\cdot
{\pi_{(\underline{i-1},i)}}_{\ast} \circ \pi_{(i-1,\underline{i})}^{\ast}( z^i_{n-i-d+j-1}).\end{equation}
Since $j-d+i-1<0$, the cycle ${\pi_{(\underline{i-1},i)}}_{\ast} \circ \pi_{(i-1,\underline{i})}^{\ast}( Z^i_{n-i-d+j-1})$ is trivial by dimensional reasons.

If $j-d+i-1=0$, combining (2.7), (2.8) and \cite[Proposition 2.1]{uINV}, one get that $c_{d-i+1}(T_{i-1})$ is congruent modulo $2$ to
\[{\pi_{(\underline{i-1},i)}}_{\ast} \circ \pi_{(i-1,\underline{i})}^{\ast}
(Z^{i}_{n-3i+d+1})+W^{i-1}_{d-i+1}+
{\pi_{(\underline{i-1},i)}}_{\ast} \circ \pi_{(i-1,\underline{i})}^{\ast}(W^i_{d-i}\cdot Z^i_{n-2i+1})\]
and the two last summands of the previous cycle are equal modulo $2$ by identity (2.9) (valid for any $1 \leq j \leq d$)
and the base of the induction.

Otherwise -- if $j-d+i-1>0$ -- one proceeds similarly combining (2.7) with (2.8)
and using identity (2.9).

\end{proof}

\section{Proof of the first part of Main Theorem}

We use notations and materials introduced in sections 1 and 2.
For any $i \in \{1,\dots, d\}$, we denote by $\theta_i$ the class of the subvariety 
\[\{(y,x_1,\dots,x_{i+1})\;|\;x_1,\dots x_{i+1} \in y\}\subset G_i \times X^{i+1}\]
in $\text{CH}(G_i\times X^{i+1})$. Viewing the cycle $\theta_i$ as a correspondence
$G_i\rightsquigarrow  X^{i+1}$,
we set 
\[\alpha_i:= (\theta_i)_{\ast}(Z^i_{n-i})+\rho_i   \in\text{CH}(X_{K}^{i+1}),\]
and we view $\alpha_i$ as a correspondence $X_{K}\rightsquigarrow X^{i}_{K}$.

\medskip

The following computations are the key point in the proof of the first part of Theorem\;1.1
(see Corollary 3.15).

\begin{prop}
\textit{For any $i \in \{1,\dots, d\}$, one has}
\[\left( \alpha_i \;(\text{mod}\;2)\right)_{\ast}(h^k)=\left\{\begin{array}{ll} 
\text{sym}\left((\times_{j=1}^{i-1}h^j)\times 1\right) & \textit{if} \:\: k=0  \textit{;}\\
0 & \textit{if}\:\: 1\leq k\leq i-1 \textit{;} \\
\text{sym}\left((\times_{j=1}^{i-1}h^j)\times h^k\right) & \textit{if}\:\: i \leq k \leq d \textit{.}
\end{array} \right. \]
\end{prop}

\begin{proof}
The following lemma provides an appropriate formula
for $\left((\theta_i)_{\ast}(Z^i_{n-i})\right)_{\ast}$.
We write $p$ with underlined target for projections. 

\begin{lemme}
\textit{For any} $x\in \text{CH}(X_{K})$\textit{, one has}
\[\left((\theta_i)_{\ast}(Z^i_{n-i})\right)_{\ast}(x)=
{p_{G_i\times \underline{X^{i}}}}_{\;\ast}\left( p_{\underline{G_i}\times X^{i}}^{\ast}\left(
{\pi_{(0,\underline{i})}}_{\ast} \circ \pi_{(\underline{0},i)}^{\ast}(x)   \cdot Z^i_{n-i}\right) \cdot \eta_i  \right),\]
\textit{where $\eta_i$ is the class in} $\text{CH}(G_i\times X^{i})$ \textit{of the subvariety
$\{(y,x_1,\dots,x_i)\;|\; x_1,\dots, x_i\in y\}\subset G_i\times X^i$.}
\end{lemme}

\begin{proof}

By definition, one has
\[(\theta_i)_{\ast}(Z^i_{n-i})={p_{G_i\times \underline{X^{i+1}}}}_{\;\ast}\left((Z^i_{n-i} \times [X^{i+1}])\cdot \theta_i\right)\]
so 
\[\left((\theta_i)_{\ast}(Z^i_{n-i})\right)_{\ast}(x)={p_{X\times \underline{X^i}}}_{\;\ast}\left((x\times [X^i])\cdot{p_{G_i\times \underline{X^{i+1}}}}_{\;\ast}\left((Z^i_{n-i} \times [X^{i+1}])\cdot \theta_i\right)\right).\]
Therefore, using the Projection Formula (see \cite[Proposition 56.9]{EKM}) and the following fiber product diagram with respect to projections
\[\xymatrix{
G_i\times X \times X^i \ar[r] \ar[d] & X \times X^i \ar[d] \\
G_i \times X^i \ar[r]  & X^i
},\]
one get that 
\[\left((\theta_i)_{\ast}(Z^i_{n-i})\right)_{\ast}(x)=
{p_{G_i\times \underline{X^{i}}}}_{\;\ast}\circ {p_{\underline{G_i}\times X\times \underline{X^i}}}_{\;\ast}
\left((Z^i_{n-i}\times x \times [X^i]) \cdot \theta_i  \right).\]
Moreover, the cycle $\theta_i$ can be rewritten as
\[\theta_i=p_{\underline{G_i\times X}\times X^{i}}^{\ast}\left([\mathcal{F}(i,0)]\right)\cdot
 p_{\underline{G_i}\times X\times \underline{X^i}}^{\ast}(\eta_i),\]
where $[\mathcal{F}(i,0)]$ is the class in $\text{CH}(G_i\times X)$ of the subvariety
$\mathcal{F}(i,0)\subset G_i\times X$.
Hence, using the Projection Formula twice, one get that 
\begin{align*}
\left((\theta_i)_{\ast}(Z^i_{n-i})\right)_{\ast}(x)  & =
 {p_{G_i\times \underline{X^{i}}}}_{\;\ast}\left(\left({p_{\underline{G_i}\times X}}_{\ast}
\left((Z^i_{n-i} \times x) \cdot [\mathcal{F}(i,0)]\right)\times [X^i]\right)\cdot \eta_i  \right) \\
 & = {p_{G_i\times \underline{X^{i}}}}_{\;\ast}\left(   \left(\left(
 {p_{\underline{G_i}\times X}}_{\ast}
 \left(p_{G_i\times \underline{X}}^{\ast}(x) \cdot [\mathcal{F}(i,0)]    \right)\cdot Z^i_{n-i}\right)\times [X^i]\right)\cdot \eta_i  \right) .
\end{align*}

Furthermore, by denoting the closed embedding $\mathcal{F}(i,0)\hookrightarrow G_i\times X$ as $in$,
one has $in_{\ast}(1)= [\mathcal{F}(i,0)]$. It follows again from the Projection Formula that 
$p_{G_i\times \underline{X}}^{\ast}(x)\cdot [\mathcal{F}(i,0)]=in_{\ast}\circ in^{\ast}\left(p_{G_i \times \underline{X}}^{\ast}(x)\right)$. 
Consequently, one has
\[{p_{\underline{G_i}\times X}}_{\ast}\left(p_{G_i \times \underline{X}}^{\ast}(x)\cdot [\mathcal{F}(i,0)] \right)
={\pi_{(0,\underline{i})}}_{\ast} \circ \pi_{(\underline{0},i)}^{\ast}(x) \]
and the lemma is proven.
\end{proof}

For any $x\in \text{CH}^k(X_{K})$ with $k\leq i-1$, the cycle 
${\pi_{(0,\underline{i})}}_{\ast} \circ \pi_{(\underline{0},i)}^{\ast}(x)$
is trivial by dimensional
reasons. Thus, by Lemma 3.2, the cycle $\left((\theta_i)_{\ast}(Z^i_{n-i})\right)_{\ast}(x)$ is also trivial.
Therefore, since $(\rho_i)_{\ast}(h^k)=\text{sym}\left((\times_{j=1}^{i-1}h^j)\times 1\right)$ for $k=0$ and is trivial for $0<k\leq d$,
one get the conclusion of Proposition 3.1 for the cases
 $k\leq i-1$.

\medskip

Note that for $i\leq k\leq d$, Lemma 3.2 provides the following identity
\begin{equation}\left( \alpha_i \;(\text{mod}\;2)\right)_{\ast}(h^k)=
{p_{G_i\times \underline{X^{i}}}}_{\;\ast}\left( 
w^i_{k-i}   \cdot z^i_{n-i} \cdot \eta_i \right) \; \text{in}\; \text{Ch}(X^i_K),
\end{equation}
where we abuse notation and write $\eta_i$ for $\eta_i \;(\text{mod}\;2)$.

\medskip

We prove the cases $i\leq k\leq d$ of Proposition 3.1 by backward induction on $i$.

\medskip

For any $1\leq i \leq d$, by the very definition, one has
\begin{equation}
\eta_i= \prod_{j=1}^i  \left(\text{Id}_{G_i}\times p_{X^i_j} \right)^{\ast}\left([\mathcal{F}(i,0)]\right) \;\text{in}\; \text{CH}(G_i\times X^i),
\end{equation}
with $p_{X^i_j}$ the projection from $X^i$ to the $j$-th coordinate.
Since $X_K$ is cellular, the cycle $[\mathcal{F}(i,0)]$ decomposes as 
\begin{equation} [\mathcal{F}(i,0)]=\sum_{m=0}^d z^i_{n-i-m}\times h^m +
\sum_{m=i}^d w^i_{m-i}\times l_m  \;\text{in}\; \text{Ch}({G_i}_K\times X_K),
\end{equation}
where $l_d$ has to be replaced by the other class $l_d'$ of maximal totally isotropic subspaces if
$n=2d$ and $l_d^2$ is not zero (i.e. if $4$ divides $n$).
\medskip

The base of the backward induction $i=d$ (so $i=k=d$) is obtained by combining 
the identities (3.3), (3.4)  and (3.5) for $i=d$ (recall also that $w^i_0=1$ for any $0\leq i \leq d$ by Lemma 2.4(i)) 
with the fact that, for any integers $0\leq a_0\leq a_1 \leq \cdots \leq a_{e} \leq d$, with $e\leq d$, one has
\[\text{deg}\left(\prod_{j=0}^e z^d_{n-d-a_j}    \right)=\left \{\begin{array}{ll} 1 & \;\;\text{if}\;\;              \{a_0, a_1, \dots , a_e\}=\{0,1,\dots , d\}\; ; \\
 0 &\;\; \text{otherwise} ,\end{array}\right.\]
where $\text{deg}:\text{Ch}({G_d}_K)\rightarrow \text{Ch}(\text{Spec}(K))=\mathbb{Z}/2\mathbb{Z}$ is 
the homomorphism associated with the push-forward of the structure morphism, see \cite[Lemma 87.6]{EKM}.

\medskip

The backward induction step will follow from Lemma 3.10, which make use of the following statement.

\begin{lemme}
\textit{For any $2\leq i\leq d$, one has}
\[
 z^i_{n-i} \cdot \eta_i
= \left( \pi_{(i-1,\underline{i})} \times \text{Id}_{X^i} \right)_{\ast}\left(
\left( \pi_{(i-1,\underline{i})} \times p_{X^i_i} \right)^{\ast}\left([\mathcal{F}(i,0)] \right)
\cdot \left(\pi_{(\underline{i-1},i)}\times p_{X^i_{\overline{i}}}\right)^{\ast}
(z^{i-1}_{n-i+1}\cdot \eta_{i-1})
\right) .\]
\end{lemme}

\begin{proof}
It follows from the identity (3.5) and Lemma 2.4(ii), (iii) and (iv) that
\begin{equation}
\left(\pi_{(\underline{i-1},i)}\times \text{Id}_{X}\right)^{\ast}\left([\mathcal{F}(i-1,0)]\right)=
\left(c_1(\mathcal{O}(1))\times 1 + 1\times h \right)\cdot
\left(\pi_{(i-1,\underline{i})}\times \text{Id}_{X}\right)^{\ast}\left([\mathcal{F}(i,0)]\right)
\end{equation}
in $\text{Ch}(\mathcal{F}(i-1,i)_K \times X_K)$.

Moreover, by (3.4), one has
\begin{equation}
\left(\pi_{(\underline{i-1},i)}\times p_{X^i_{\overline{i}}}\right)^{\ast}
(z^{i-1}_{n-i+1}\cdot \eta_{i-1})=
\pi_{(\underline{i-1},i)}^{\ast}(z^{i-1}_{n-i+1})\cdot \prod_{j=1}^{i-1}
\left( \pi_{(\underline{i-1},i)} \times p_{X^i_j} \right)^{\ast}\left([\mathcal{F}(i-1,0)] \right)
\end{equation}
with $p_{X^i_{\overline{i}}}$ the projection from $X^i$ to the $i-1$ first coordinates.
By (3.7), the right member of the previous equation can be rewritten as
\[
\pi_{(\underline{i-1},i)}^{\ast}(z^{i-1}_{n-i+1})\cdot \prod_{j=1}^{i-1}
\left(
c_1(\mathcal{O}(1))\times [X^i] + 1\times p_{X^i_j}^{\ast}(h)\right)
\cdot \prod_{j=1}^{i-1}
\left( \pi_{(i-1,\underline{i})} \times p_{X^i_j} \right)^{\ast}\left([\mathcal{F}(i,0)] \right)
.\]
Hence, by multiplying the equation (3.8) by 
$\left( \pi_{(i-1,\underline{i})} \times p_{X^i_i} \right)^{\ast}\left([\mathcal{F}(i,0)] \right)$ 
and using Lemma 2.4(ii), one get
\begin{multline}
\left( \pi_{(i-1,\underline{i})} \times p_{X^i_i} \right)^{\ast}\left([\mathcal{F}(i,0)] \right)
\cdot \left(\pi_{(\underline{i-1},i)}\times p_{X^i_{\overline{i}}}\right)^{\ast}
(z^{i-1}_{n-i+1}\cdot \eta_{i-1}) =  \\
 \left(c_1(\mathcal{O}(1))\times [X^i]\right) \cdot  \prod_{j=1}^{i-1}
\left(
c_1(\mathcal{O}(1))\times [X^i] + 1\times p_{X^i_j}^{\ast}(h)\right)
\cdot \left( \pi_{(i-1,\underline{i})} \times \text{Id}_{X^i} \right)^{\ast}( z^i_{n-i} \cdot \eta_i)
.
\end{multline}
Furthermore, it follows from the Projective Bundle Theorem (see \cite[Theorem 53.10]{EKM})
applied to ${ \pi_{(i-1,\underline{i})}}: \mathcal{F}(i-1,i)\rightarrow G_i$ that 
${ \pi_{(i-1,\underline{i})}}_{\ast}(c_1(\mathcal{O}(1))^i)=1$.
Since $\text{dim}\;\mathcal{F}(i-1,i)-\text{dim}\;G_i=i$,
one obtains the conclusion of the lemma by composing the 
equation (3.9) by $\left( \pi_{(i-1,\underline{i})} \times \text{Id}_{X^i} \right)_{\ast}$
and using the Projection Formula.

\end{proof}

\begin{lemme}
\textit{For any $2\leq i\leq d$ and $i\leq k \leq d$, the cycle in identity (3.3) can be rewritten as}
\[
\sum_{m=0}^k \sum_{j=\text{max}(i-m,0)}^{\text{min}(k-m,i)}
 {p_{G_{i-1}\times \underline{X^{i-1}}}}_{\;\ast}\left( 
w^{i-1}_{k-m-j}\cdot \sigma_{i-1}^j   \cdot z^{i-1}_{n-i+1} \cdot \eta_{i-1}  \right)\times h^m ,\]
\textit{with} $\sigma_{i-1}^j={\pi_{(\underline{i-1},i)}}_{\ast} \circ \pi_{(i-1,\underline{i})}^{\ast}
(z^{i}_{n-2i+j}) \in \text{Ch}^j({G_{i-1}}_K)$.
\end{lemme}

\begin{proof}
It follows from Lemma 3.6 and the Projection Formula that the cycle in identity (3.3) can be rewritten as
\[
{p_{G_{i-1}\times \underline{X^{i}}}}_{\;\ast}\left( 
\left( \pi_{(\underline{i-1},i)} \times \text{Id}_{X^i} \right)_{\ast}\circ
\left( \pi_{(i-1,\underline{i})} \times p_{X^i_i} \right)^{\ast}\left(w^i_{k-i}\cdot [\mathcal{F}(i,0)] \right)
\cdot z^{i-1}_{n-i+1} \cdot (\eta_{i-1}\times 1) \right)
.\]
If $k<d$ then, by decomposition (3.5) and dimensional considerations, the latter cycle is equal to
\[\sum_{m=0}^k {p_{G_{i-1}\times \underline{X^{i-1}}}}_{\;\ast}\left(
 {\pi_{(\underline{i-1},i)}}_{\ast} \circ \pi_{(i-1,\underline{i})}^{\ast}(w^i_{k-i}\cdot z^i_{n-i-m})
\cdot z^{i-1}_{n-i+1} \cdot \eta_{i-1}
\right)\times h^m
\]
and, if $k=d$, one also has to consider the extra term
\[{p_{G_{i-1}\times \underline{X^{i-1}}}}_{\;\ast}\left(
 {\pi_{(\underline{i-1},i)}}_{\ast} \circ \pi_{(i-1,\underline{i})}^{\ast}(w^i_{d-i}\cdot w^i_{d-i})
\cdot z^{i-1}_{n-i+1} \cdot \eta_{i-1}
\right)\times l_d\] (or $l_d'$),
 which is trivial since ${\pi_{(\underline{i-1},i)}}_{\ast} \circ \pi_{(i-1,\underline{i})}^{\ast}(w^i_{d-i}\cdot w^i_{d-i})=0$. 
Indeed, by Lemma 2.4(iii), the Projection Formula and dimensional reasons, ${\pi_{(\underline{i-1},i)}}_{\ast} \circ \pi_{(i-1,\underline{i})}^{\ast}(w^i_{d-i}\cdot w^i_{d-i})$ is equal
to \[{\pi_{(\underline{i-1},i)}}_{\ast} \left( \pi_{(i-1,\underline{i})}^{\ast}(w^i_{d-i-1})\cdot c_1(\mathcal{O}(1))
\cdot \pi_{(i-1,\underline{i})}^{\ast}(w^i_{d-i})\right),\]
which can be rewritten as 
${\pi_{(\underline{i-1},i)}}_{\ast} \left( \pi_{(i-1,\underline{i})}^{\ast}(w^i_{d-i-1})
\cdot \pi_{(\underline{i-1},i)}^{\ast}(w^{i-1}_{d-i+1})\right)$ (Lemma 2.4(iv)) and therefore is zero by the Projection Formula and dimensional reasons.
Lemma 2.5 completes the proof.
\end{proof}

Let $2\leq i\leq d$. On the one hand, by identity (3.3) and backward induction hypothesis, for any $i\leq k\leq d$, one has
\begin{equation} {p_{G_i\times \underline{X^{i}}}}_{\;\ast}\left( 
w^i_{k-i}   \cdot z^i_{n-i} \cdot \eta_i \right)
=\text{sym}\left((\times_{j=1}^{i-1}h^j)\times h^k\right).\end{equation}
Therefore, in particular, the coordinate of (3.11) on top right $h^{i-1}$, i.e., 
\[{p_{\underline{X^{i-1}}\times X}}_{\;\ast}\left(\left({p_{G_i\times \underline{X^{i}}}}_{\;\ast}\left( 
w^i_{k-i}   \cdot z^i_{n-i} \cdot \eta_i \right)\right)\cdot [X^{i-1}]\times l_{i-1}\right)\]
is equal to
\begin{equation}\text{sym}\left((\times_{j=1}^{i-2}h^j)\times h^k\right).\end{equation}
On the other hand, by Lemma 3.10, this coordinate is also equal to
\begin{equation}\sum_{j=1}^{\text{min}(k-i+1,i)}
 {p_{G_{i-1}\times \underline{X^{i-1}}}}_{\;\ast}\left( 
w^{i-1}_{k-i+1-j}\cdot \sigma_{i-1}^j   \cdot z^{i-1}_{n-i+1} \cdot \eta_{i-1}  \right).\end{equation}
Since 
$W^{i-1}_{k-i+1-j}=c_{k-i+1-j}(V\mathbbm{1}/T_{i-1})$
(see \cite[Proposition 2.1]{uINV}) and $\sigma_{i-1}^j =$$c_j(T_{i-1})$$\;(\text{mod}\;2)$ (see Lemma 2.6), by Whitney Sum Formula, one has
\[\sum_{j=0}^{\text{min}(k-i+1,i)}w^{i-1}_{k-i+1-j}\cdot \sigma_{i-1}^j  
=\sum_{j=0}^{\text{min}(k-i+1,i)} c_{k-i+1-j}(V\mathbbm{1}/T_{i-1})\cdot c_j(T_{i-1})
=c_{k-i+1}(V\mathbbm{1})=0.\]
Consequently, in view of (3.12) and (3.13), one get
\begin{equation}{p_{G_{i-1}\times \underline{X^{i-1}}}}_{\;\ast}\left( 
w^{i-1}_{k-i+1} \cdot z^{i-1}_{n-i+1} \cdot \eta_{i-1}  \right)=\text{sym}\left((\times_{j=1}^{i-2}h^j)\times h^k\right).\end{equation}
Note that, by identities (3.3) and (3.14), it only remains to prove that
\[{p_{G_{i-1}\times \underline{X^{i-1}}}}_{\;\ast}\left( 
 z^{i-1}_{n-i+1} \cdot \eta_{i-1}  \right)=\text{sym}\left(\times_{j=1}^{i-1}h^{j}\right)\]
to complete the backward induction step.
On the one hand, by backward induction hypothesis, 
the coordinate of (3.11) on top right $h^{k}$ is
\[\text{sym}\left(\times_{j=1}^{i-1}h^{j}\right)\] 
and
on the other hand, by Lemma 3.10, it is also equal to 
\[{p_{G_{i-1}\times \underline{X^{i-1}}}}_{\;\ast}\left( 
 z^{i-1}_{n-i+1} \cdot \eta_{i-1}  \right).\]

Proposition 3.1 is proven.

\end{proof}

We are now able to prove the first part of the main result of this note (Theorem 1.1), which we restate below.

\begin{cor}
 \textit{Let $i \in \{0,\dots, d\}$. If $z^i_{n-i}$ is rational then} $\rho_i\;(\text{mod}\;2)$ \textit{is also rational.}
\end{cor}

\begin{proof}
Since the conclusion is obvious for $i=0$, we assume that $i\geq 1$.
In view of the ring structure of $\text{CH}(X_{K}^{i+1})$ (see \cite[\S 68]{EKM}), and
knowing that the cycle $\alpha_i$ is symmetric, one
deduces from Proposition 3.1 that
\[\alpha_i\;(\text{mod}\;2)=\Delta_i\;(\text{mod}\;2)+\beta\]
with $\beta$ a sum of nonessential elements (a nonessential element is an external product of powers of the hyperplane class, it is always rational).
Since $\alpha_i= (\theta_i)_{\ast}(Z^i_{n-i})+\rho_i$ and $\Delta_i\;(\text{mod}\;2)$ is rational (Lemma 2.1), the corollary is proved. 
\end{proof}

The following statement is a consequence of Proposition 3.1 and its proof.

\begin{prop}
\textit{Let $1\leq i\leq d-1$, $i+1\leq k \leq d$ and $1\leq m\leq i$.
For any integers $0\leq a_1\leq a_2 \leq \cdots \leq a_i \leq d$, the integer}
\[\text{deg}\left(\left(Z^i_{n-i}\cdot \prod_{l=1}^i Z^i_{n-i-a_l}\right) \cdot \left(
\sum_{j=0}^{i-m}W^i_{k-m-j}\cdot c_j(T_i)  \right)  \right)\]
\textit{is congruent to} $1\,(\text{mod}\,2)$ \textit{if $\{a_1, \dots , a_i\}=\{k\} \cup (\{1,2,\dots , i\}\backslash \{m\})$
and to} $0\,(\text{mod}\,2)$ \textit{otherwise.}
\end{prop}

\begin{proof}
Proposition 3.1 and Lemma 3.10 provide two descriptions of the coordinate of the cycle (3.3) on top right $h^m$ (note that one can decompose the cycle $\eta_i\in \text{Ch}({G_i}_K\times X^i_K)$ appearing in Lemma 3.10 as a sum of external products by combining equations (3.4) and (3.5)).
The conclusion is obtained by comparing these two descriptions and applying the Whitney Sum Formula
to that given by Lemma 3.10.
\end{proof}

\begin{rem}
We also retrieve \cite[Statement 2.15]{uINV} by proceeding the same way as in the previous proof but considering the coordinate
on top right $h^k$ instead of top right $h^m$.
Proposition 3.16 is completed for $i=d$ by \cite[Lemma 87.6]{EKM}.
\end{rem}

\section{Elementary Discrete Invariant and first Witt index}

In this section, we continue to use notations and materials introduced in the previous sections.
Corollary 3.15 implies the following condition on the first Witt index $i_1$ of the quadric $X$.

\begin{prop}
\textit{Assume $X$ anisotropic and let $i \in \{1,\dots, d\}$. If $z^i_{n-i}$ is rational then $i_1\leq i$.}
\end{prop}

\begin{proof}
In view of Corollary 3.15, the statement follows from the next lemma.

\begin{lemme}
\textit{Assume $X$ anisotropic and let $i \in \{1,\dots, d\}$. If} $\rho_i\;(\text{mod}\;2)$ \textit{is rational then $i_1\leq i$.}
\end{lemme}

\begin{proof}
Let $i \in \{1,\dots, d\}$. Suppose that $\rho_i\;(\text{mod}\;2)$ is rational and that $i_1> i$.
We claim that this implies that $\rho_{i-1}\;(\text{mod}\;2)$ is also rational.

Indeed, let us denote by $\pi \in \text{Ch}^{n-i_1+1}(X_K^2)$ 
the \textit{$1$-primordial cycle} 
(see \cite[Definition 73.16]{EKM} and paragraph
right after \cite[Theorem 73.26]{EKM}, it is a rational cycle).
Even if it means adding a rational cycle to $\pi$, one can assume that $\pi$
decomposes as 
\[\pi=1\times l_{i_1-1}+l_{i_1-1}\times 1+\sum_{j=i_1}^{d-i_1+1}a_j\left( h^j \times l_{j+i_1-1}
+ l_{j+i_1-1}\times h^j\right)\]
for some $a_j\in \mathbb{Z}/2\mathbb{Z}$ (the fact that one can choose to make the previous sum start from $j=i_1$
is due to \cite[Proposition 73.27]{EKM}).
Since $i_1> i$, the computation of the composition of rational correspondences
$\left(\rho_i\;(\text{mod}\;2)\right)_K\circ \left((1\times h^{i_1-i})\cdot \pi\right)$ gives the identity
\[\left(\rho_i\;(\text{mod}\;2)\right)_K\circ \left((1\times h^{i_1-i})\cdot \pi\right)=1\times \left(\rho_{i-1}\;(\text{mod}\;2)\right)_K.\]
Therefore, pulling back the latter algebraic cycle with respect to the diagonal morphism
\[\begin{array}{rll}
X^i &\longrightarrow &X^{i+1} \\
(x_1,x_2,x_3,\dots,x_i) & \longmapsto & (x_1,x_1,x_2,x_3\dots,x_i)
\end{array}\]
(for example), one get that $\rho_{i-1}\;(\text{mod}\;2)$ is also a rational cycle.

It follows from the claim that $\rho_{1}\;(\text{mod}\;2)$ is rational. By \cite[Remark 80.9]{EKM}, this implies that $i_1=1$.
\end{proof}
The proposition is proved.
\end{proof}



\begin{rem}
Combining 
Proposition 4.1 with \cite[Propositon 2.5]{uINV}, one get that, representing 
the \textit{Elementary Discrete Invariant} $EDI(X)$ of the quadric $X$ as a 
$(d+1)\times (d+1)$ coordinate square (see \cite[Definition 2.3 and the paragraph right after]{uINV}), there is no marked integral node below
the $i_1$-th diagonal of this square. That is to say, the invariant $EDI(X)$ looks as

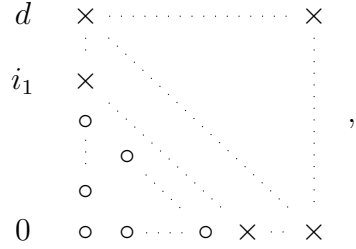
\[\begin{tikzpicture}[baseline=(current bounding box.center)]
\matrix (m) [matrix of math nodes,nodes in empty cells ]{
d\;\; & \times  & & & & & & & \times  \\
  & & & & & & & &     \\
i_1 \;\; & \times & & & & &  &  &   \\
 & \circ & & & & & &  &    \\
  & & \circ  & & & & & &     \\
  & \circ  & & & & & & &  \\
0\;\; & \circ &\circ & & &\circ &\times &  & \times \\
} ;

\draw[loosely dotted] (m-4-2)-- (m-6-2);

\draw[loosely dotted] (m-7-3)-- (m-7-6);
\draw[loosely dotted] (m-5-3)-- (m-7-6);

\draw[loosely dotted] (m-3-2)-- (m-7-7);

\draw[loosely dotted] (m-1-2)-- (m-3-2);
\draw[loosely dotted] (m-1-2)-- (m-1-9);
\draw[loosely dotted] (m-1-9)-- (m-7-9);
\draw[loosely dotted] (m-1-2)-- (m-7-9);
\draw[loosely dotted] (m-7-7)-- (m-7-9);

\end{tikzpicture},\]
where $\circ$ means that the node is unmarked.

\noindent We recall that each row is associated with an orthogonal grassmannian, starting with the quadric at the bottom row, and that each integral node corresponds to an \textit{elementary classes} as defined by A.\,Vishik in \cite{uINV}, with codimension decreasing from left to right (the cycles $z^i_{n-i}\in \text{Ch}^{n-i}({G_i}_{K})$ correspond to the first column on the left).
\end{rem}

\section{Proof of the second part of Main Theorem}

We use notations and materials introduced in previous sections.
The fact that if $\rho_1$ is rational then $Z^1_{n-1}$ is also rational has been shown by A.\,Vishik in the first part of the proof of \cite[Theorem 4.4]{sov}.

\begin{prop}
 \textit{Let $i \in \{0,\dots, d\}$. If} $\rho_i\;(\text{mod}\;2)$ \textit{is rational then $z^i_{n-i}$ is also rational.}
\end{prop}

\begin{proof}
Since the conclusion is obvious for $i=0$, we assume that $i\geq 1$.
Consider the rational cycle
\[\theta'_i:=p_{\underline{X}\times X \times \underline{\mathcal{F}(0,i)}}^{\ast}\circ
\left(\text{Id}_X \times \pi_{(\underline{0},i)}\right)^{\ast}\left(\Delta_1\;(\text{mod}\;2)\right)\cdot
p_{X\times \underline{X} \times \underline{\mathcal{F}(0,i)}}^{\ast}\circ
\left(\text{Id}_X \times \pi_{(0,\underline{i})}\right)^{\ast}\left([\mathcal{F}(0,i)]_K\right)
\]
in $\text{Ch}^{2n-i}(X^2_K\times \mathcal{F}(0,i)_K)$, with 
$\Delta_1\in \text{CH}^{n}(X^{2}_K)$ the cycle introduced in \S 2.1.
We view $\theta'_i$ as a correspondence $X^2_K \rightsquigarrow \mathcal{F}(0,i)_K$.
Using decomposition (3.5), one get that
for any $(\alpha,\beta)\in 
\{1,h,\dots,h^{i-1},l_0\}^2$, with $\alpha\neq \beta$, one has
\[{\theta'_i}_{\ast}(\alpha\times \beta)=\left\{\begin{array}{ll} 
\pi_{(\underline{0},i)}^{\ast}(h^k)\cdot \pi_{(0,\underline{i})}^{\ast}(z^i_{n-i}) & \text{if}\:\: \alpha\times \beta =h^k\times l_0 \; ; \\
0 & \text{otherwise.} 
\end{array} \right.\]
Consequently, since $\rho_i\;(\text{mod}\;2)$ is assumed to be rational, one obtains that the cycle
\[\left(\text{Id}_{X^{i-1}} \times  {\theta'_i} \right)_{\ast}(\rho_i\;(\text{mod}\;2))=
\sum_{k=0}^{i-1}
\text{sym}\left(\times_{\substack{ j=0 \\ j\neq k}}^{i-1}
h^{j}\right) \times \left(\pi_{(\underline{0},i)}^{\ast}(h^k)\cdot\pi_{(0,\underline{i})}^{\ast}(z^i_{n-i})\right)\]
is rational. 
By multiplying the previous cycle by $[X^{i-1}]\times \pi_{(\underline{0},i)}^{\ast}(h)$
and then composing by $\left(\text{Id}_{X^{i-1}} \times \pi_{(0,\underline{i})}\right)_{\ast}$, one get, 
using the Projection Formula and Lemma 2.4(i), that the cycle
\[\text{sym}\left(\times_{j=0}^{i-2}h^{j}\right) \times z^i_{n-i}\]
is rational.
Therefore, it suffices to prove that, for any $2\leq k\leq i-1$, the rationality of $\text{sym}\left(\times_{j=0}^{i-k}h^{j}\right) \times z^i_{n-i}$ implies the rationality of $\text{sym}\left(\times_{j=0}^{i-k-1}h^{j}\right) \times z^i_{n-i}$ to conclude.
This follows from the next formula (which uses decomposition (3.5))

\begin{multline*}
{p_{\underline{X^{i-k}}\times X \times \underline{G_i}}}_{\ast}\left( [X^{i-k}]\times (h^k \times [G_i] \cdot [\mathcal{F}(0,i)]_K ) \cdot
\text{sym}\left(\times_{j=0}^{i-k}h^{j}\right) \times z^i_{n-i} \right)  = \\
  \text{sym}\left(\times_{j=0}^{i-k-1}h^{j}\right) \times z^i_{n-i}   \; . 
\end{multline*}

\end{proof}

\bibliographystyle{acm} 
\bibliography{references}

\begin{thebibliography}{1}

\bibitem{EKM}
{\sc Elman, R., Karpenko, N., and Merkurjev, A.}
\newblock {\em The algebraic and geometric theory of quadratic forms}, vol.~56
  of {\em American Mathematical Society Colloquium Publications}.
\newblock American Mathematical Society, Providence, RI, 2008.

\bibitem{sov}
{\sc Vishik, A.}
\newblock Symmetric operations (in russian).
\newblock In {\em Trudy Mat. Inst. Steklova 246\/} (2004), Algebr. Geom.
  Metody, Svyazi i Prilozh., pp.~92--105.
\newblock English transl: Proc. of the Steklov Institute of Math. 246, 79--92
  (2004).

\bibitem{GPQ}
{\sc Vishik, A.}
\newblock Generic points of quadrics and {Chow} groups.
\newblock {\em Manuscripta Math 122}, 3 (2007), 365--374.

\bibitem{uINV}
{\sc Vishik, A.}
\newblock Field of $u$-invariant $2^r+1$.
\newblock In {\em Algebra, arithmetic, and geometry: in honor of Yu. I. Manin.
  Vol. II}, vol.~270 of {\em Progr.Math}. Birkh{\"a}user Boston Inc., Boston,
  MA, 2009, pp.~661--685.

\end{thebibliography}
\end{document}